\documentclass[10pt, a4paper]{amsart}
\usepackage{geometry} 
\usepackage{url}
\usepackage{palatino}
\usepackage{amsthm}
\usepackage{amssymb}
\usepackage{latexsym}
\usepackage{amsfonts}
\usepackage{amsmath}
\usepackage{amstext}
\usepackage{accents}
\usepackage{cancel}
\usepackage{color}
\usepackage{soul}
\usepackage{enumerate}
\usepackage{eucal}
\usepackage{amscd}
\usepackage{hyperref}
\usepackage{graphicx}
\usepackage{verbatim}
\usepackage{caption}
\usepackage{enumitem}

\usepackage[usenames]{xcolor}
\setstcolor{red}

\definecolor{greenish}{RGB}{50,160,0}

\newtheorem{theorem}{Theorem}[section]

\theoremstyle{definition}
\newtheorem{definition}[theorem]{Definition}
\newtheorem{example}[theorem]{Example}

\theoremstyle{remark}
\newtheorem{remark}[theorem]{Remark}

\numberwithin{equation}{section}

\begin{document}
\title[The uniqueness of the Enneper surfaces and Chern-Ricci functions on 
minimal surfaces]{The uniqueness of the Enneper surfaces and \\ Chern--Ricci functions on 
minimal surfaces}

\author[Hojoo Lee]{Hojoo Lee}
\address{Center for Mathematical Challenges, Korea Institute for Advanced Study, Hoegiro 85, Dongdaemun-gu, Seoul 02455, Korea}
\email{momentmaplee@gmail.com}

\begin{abstract}
We construct the first and second Chern-Ricci functions on negatively curved minimal surfaces in ${\mathbb{R}}^{3}$ using Gauss curvature 
and angle functions, and establish that they become harmonic functions on the minimal surfaces. We prove that a minimal surface has constant first Chern-Ricci function if and only if it is Enneper's surface. We explicitly determine the moduli space of minimal surfaces having constant second Chern-Ricci function, which contains catenoids, helicoids, and their associate families. 
\end{abstract}

\keywords{Enneper's surface, holomorphic quadratic differentials, minimal surfaces, Ricci condition}
\subjclass{53A10, 49Q05}
 
 \maketitle

   \begin{center}
  {\small{\textit{To my mother who always gives me courage}}}
   \end{center}

 \bigskip

 \section{Enneper's surfaces and other minimal surfaces in ${\mathbb{R}}^{3}$}
  
  What is the Enneper surface? The Enneper-Weierstrass representation \cite{Oss86} says that a simply connected minimal surface in 
  ${\mathbb{R}}^{3}$ can be constructed by the conformal harmonic mapping 
  \[
  {\mathbf{X}}(\zeta) =   {\mathbf{X}}({\zeta}_{0}) + \left(  \textrm{Re} \int_{{\zeta}_{0}}^{\zeta}  {\phi}_{1}(\zeta)  d\zeta,  \textrm{Re} \int_{{\zeta}_{0}}^{\zeta}  {\phi}_{2}(\zeta)  d\zeta,  \textrm{Re}  \int_{{\zeta}_{0}}^{\zeta}   {\phi}_{3}(\zeta)  d\zeta \right), 
\]
where the  holomorphic null curve $\phi(\zeta)$ is determined by the Weierstrass data 
$(G(\zeta), \Psi(\zeta) d\zeta):$
\[
\phi(\zeta)= \left({\phi}_{1}(\zeta), {\phi}_{2}(\zeta), {\phi}_{3}(\zeta) \right)=  \left(  \frac{1}{2}  \left(1 -G^2 \right)  \Psi, \,  \frac{i}{2} \left(1 +G^2 \right)  \Psi, \,   G  \Psi  \right)
\] 
Taking the simplest data $(G(\zeta), \Psi(\zeta) d\zeta)= (\zeta, d\zeta)$, $\zeta= u+iv \in \mathbb{C}$ yields Enneper's surface
 \begin{equation} \label{enn patch}
  {\mathbf{X}}_{{}_{\textrm{\textbf{Enn}}}}(u, v)  = \left( \frac{1}{2} \left(  u -\frac{u^3}{3} +uv^2   \right),  
  \frac{1}{2} \left( - v + \frac{v^3}{3} -u^2 v \right),   \frac{1}{2} \left( u^2 -v^2 \right) \right).
 \end{equation}
One can see the \textit{inner} rotational symmetry of its induced metric ${\mathbf{g}}_{{}_{\textrm{\textbf{Enn}}}} =  {\left( \frac{1+ {\vert \zeta \vert}^2 }{2} \right)}^{2} {\vert d\zeta \vert}^{2}$ and strictly negative Gauss curvature ${\mathcal{K}}_{{}_{\textrm{\textbf{Enn}}}} = {\mathcal{K}}_{{\mathbf{g}}_{{}_{\textrm{\textbf{Enn}}}}}  = - { \left( \frac{2}{1+{\vert \zeta \vert}^{2}  } \right)}^{4}$. We observe that 
\begin{enumerate}
\item  the conformally changed metric 
${\left( -  {\mathcal{K}}_{{}_{\textrm{\textbf{Enn}}}} \right)}^{\frac{1}{2}}  {\mathbf{g}}_{{}_{\textrm{\textbf{Enn}}}}   =  {\vert d\zeta \vert}^{2}$
is flat, and that
\item 
${\left( -  {\mathcal{K}}_{{}_{\textrm{\textbf{Enn}}}} \right)}   {\mathbf{g}}_{{}_{\textrm{\textbf{Enn}}}}   =    { \left( \frac{2}{1+{\vert \zeta \vert}^{2}  } \right)}^{2}{\vert d\zeta \vert}^{2}$ becomes the metric on the complex $\zeta$-plane  induced by the stereographic projection of the unit sphere sitting in ${\mathbb{R}}^{3}$ with respect to the north pole.
\end{enumerate}
Do these  \textit{intrinsic} properties uniquely determine Enneper's surfaces among minimal surfaces? The answer is completely no. Ricci \cite{Bl1950, BM2013, Law1970, Law1971, MM2015} showed that, if $\Sigma$ is a minimal surface in  ${\mathbb{R}}^{3}$, on non-flat points, the metric ${\left( -  {\mathcal{K}}_{ {\mathbf{g}}_{{}_{\Sigma}}} \right)}^{\frac{1}{2}}  {\mathbf{g}}_{{}_{\Sigma}}$ is flat. The Ricci condition guarantees that the metric ${\left( -  {\mathcal{K}}_{ {\mathbf{g}}_{{}_{\Sigma}}} \right)}  {\mathbf{g}}_{{}_{\Sigma}}$ has   constant Gauss curvature 1. We see that Enneper's surface becomes the simplest example illustrating Ricci conditions for negatively curved minimal surfaces in  ${\mathbb{R}}^{3}$. 

\begin{figure}
 \centering
 \includegraphics[height=4.50cm]{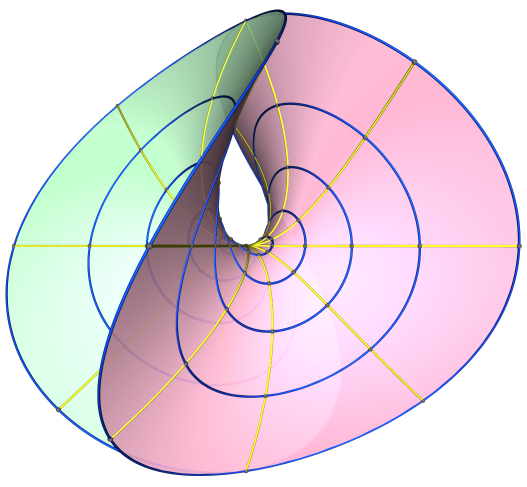}
 \caption{\small{An approximation of a part of Enneper's surface with total curvature $-4 \pi$ \; \cite{WebEnn}}} 
 \end{figure}
  We shall present the \textit{coordinate-free} characterization of the Enneper surface given by the patch (\ref{enn patch}).
To achieve this goal, we discover a geometric identity on Enneper's surface. Let's recall 
 
\begin{definition}[\textbf{Angle function on the oriented surface in ${\mathbb{R}}^{3}$}]
Given a surface  $\Sigma$ in ${\mathbb{R}}^{3}$ oriented by the unit normal vector field $\mathbf{N}$ and a constant unit vector field ${\mathbf{V}}(p)={\mathbf{V}}$ in ${\mathbb{R}}^{3}$, we introduce the angle function  $\mathbf{N}_{\mathbf{V}}:  \Sigma \to [-1, 1]$ by the formula
\[
\mathbf{N}_{\mathbf{V}} (p):= {\langle \mathbf{N}(p), {\mathbf{V}} \rangle}_{{\mathbb{R}}^{3}}, \quad p \in \Sigma,
\]
which is the normal component of the vector field $\mathbf{V}$.
\end{definition}
We observe that the induced unit normal vector field $\mathbf{N}_{{}_{\textrm{\textbf{Enn}}}}$ on the Enneper surface given by
\begin{center}
${\mathbf{X}} = {\mathbf{X}}_{{}_{\textrm{\textbf{Enn}}}}(u, v)  = \left( \frac{1}{2} \left(  u -\frac{u^3}{3} +uv^2   \right),  
  \frac{1}{2} \left( - v + \frac{v^3}{3} -u^2 v \right),   \frac{1}{2} \left( u^2 -v^2 \right) \right)$
\end{center}
  reads 
\begin{center}
$\mathbf{N}_{{}_{\textrm{\textbf{Enn}}}}(u, v) = \frac{1}{\Vert  {\mathbf{X}}_{u} \times {\mathbf{X}}_{v}  \Vert}  {\mathbf{X}}_{u} \times {\mathbf{X}}_{v}
 = \frac{1}{1+u^2 +v^2} \left( 2u, 2v, -1 +u^2 +v^2   \right)$.
\end{center}
Taking the downward vector $-{\mathbf{e}_{3}}=(0, 0, -1)$, we see that the following quantity is constant:
 \[
  { \left( -  \mathcal{K}  \right) }^{ - \frac{1}{4}}   \left(  1+ \mathbf{N}_{(-{\mathbf{e}_{3}})}  \right)  
  = \left( \frac{1+u^2 +v^2}{2} \right) \cdot  \left(  1+  \frac{1-u^2 -v^2}{ 1+u^2 +v^2} \right)  = 1.
 \]
 It turns out that the constancy of this quantity captures the geometric uniqueness of   Enneper's surfaces among  minimal surfaces 
 in ${\mathbb{R}}^{3}$, and that the geometric quantity ${ \left( -  \mathcal{K}  \right) }^{ - \frac{1}{4}}   \left(  1+ \mathbf{N}_{(-{\mathbf{e}_{3}})}  \right)$ motivates the first Chern-Ricci harmonic function on minimal surfaces.
 
\begin{theorem}[\textbf{The first Chern-Ricci harmonic function and uniqueness of Enneper's surfaces}]
Let  $\Sigma$ denote a minimal surface immersed in ${\mathbb{R}}^{3}$ with the Gauss curvature $\mathcal{K}={\mathcal{K}}_{ {\mathbf{g}}_{{}_{\Sigma}}}<0$, 
unit normal vector field $\mathbf{N}$, and  angle function $\mathbf{N}_{\mathbf{V}} \in (-1, 1]$ for a constant unit vector 
field ${\mathbf{V}}$ in ${\mathbb{R}}^{3}$. 
\begin{enumerate}
\item The first Chern-Ricci function ${\textrm{CR}}^{1}_{\mathbf{V}} := \ln \left(  { \left( -  \mathcal{K}  \right) }^{ - \frac{1}{4}}   \left(  1+ \mathbf{N}_{\mathbf{V}}  \right)    \right)$  is harmonic on $\Sigma:$
\begin{center}
       ${\triangle}_{ {\mathbf{g}}_{{}_{\Sigma}} } \,  \ln \left( \frac{1+ \mathbf{N}_{\mathbf{V}}  }{  {  \left( -  \mathcal{K}  \right) }^{\frac{1}{4}}  } \right) = 0.$ 
\end{center}
\item If the first Chern-Ricci function is constant, then it should be a part of the Enneper surface, up to  isometries and homotheties 
in ${\mathbb{R}}^{3}$. 
\end{enumerate}
\end{theorem}

 \begin{theorem}[\textbf{Harmonicity of the second Chern-Ricci function and classification of minimal surfaces with constant second Chern-Ricci function}]
Let  $\Sigma$ denote a minimal surface in ${\mathbb{R}}^{3}$ with the Gauss curvature $\mathcal{K}={\mathcal{K}}_{ {\mathbf{g}}_{{}_{\Sigma}}}<0$, 
unit normal vector $\mathbf{N}$, and angle function $\mathbf{N}_{\mathbf{V}} \in (-1, 1)$ for a constant unit vector 
field ${\mathbf{V}}$ in ${\mathbb{R}}^{3}$.
\begin{enumerate}
\item The second Chern-Ricci function ${\textrm{CR}}^{2}_{\mathbf{V}} :=\ln \left(  { \left( -  \mathcal{K}  \right) }^{ - \frac{1}{2}}   \left(  1 -  { \mathbf{N}_{\mathbf{V}} }^2 \right)    \right)$  is harmonic on $\Sigma:$
\begin{center}
  ${\triangle}_{ {\mathbf{g}}_{{}_{\Sigma}} } \,   \ln \left( \frac{1-{ \mathbf{N}_{\mathbf{V}} }^2 }{  {  \left( -  \mathcal{K}  \right) }^{\frac{1}{2}}  } \right) = 0.$ 
\end{center}
\item Any member of the moduli space of minimal surfaces of constant second Chern-Ricci function is a part of the minimal surface given by Weierstrass data $\left({\mathbf{g}}(\zeta), \frac{1}{{\mathbf{g}}'(\zeta)} d\zeta\right)=\left( e^{\alpha \left(\zeta - {\zeta}_{0} \right)},  e^{-\alpha \zeta}  d \zeta \right)$ for some constants $\alpha \in {\mathbb{C}}-\{0\}$ and ${\zeta}_{0} \in {\mathbb{C}}$. The moduli space contains catenoids, helicoids, and their associate families. 
\end{enumerate}
\end{theorem}

\begin{example}[\textbf{The second Chern-Ricci function on helicoids}]  Taking the Weierstrass data $(G(\zeta), \Psi(\zeta) d\zeta)= \left(e^{\zeta}, -i e^{-\zeta}d\zeta\right)$, $\zeta= u+iv \in \mathbb{C}$ gives ${\mathbf{X}} (u, v)= \left(- \sinh u \sin v, \sinh u \cos v, v \right)$. The second Chern-Ricci function with respect to the vector field $\mathbf{V}={\mathbf{e}_{3}}=(0, 0, 1)$ is constant:
\begin{center}
${\textrm{CR}}^{2}_{{\mathbf{e}_{3}}} =\ln \left(  { \left( -  \mathcal{K}  \right) }^{ - \frac{1}{2}}   \left(  1 -  { \mathbf{N}_{{\mathbf{e}_{3}}} }^2 \right)    \right)
= \ln \left(     {\left( \frac{  1 + {\vert e^{\zeta}\vert}^{2} }{  2 \,  \vert e^{\zeta} \vert }  \right) }^{2} \cdot  {\left( \frac{  2 \, \vert e^{\zeta} \vert }{  1 + {\vert e^{\zeta}\vert}^{2} }  \right) }^{2}  \right)= \ln 1 = 0.$
\end{center}
\end{example}

\begin{remark}
It would be interesting to generalize the Chern-Ricci harmonic functions on minimal hypersurfaces in higher dimensional Euclidean space. 
\end{remark}
 
 \section{Construction of the first and second Chern-Ricci harmonic functions}
 
 It is a classical fact that the flat points of a minimal surface are isolated, unless itself is flat everywhere. We use the geometric identities due to Chern
 and Ricci to construct harmonic functions on negatively curved minimal surfaces.
 
 \begin{theorem}[\textbf{Harmonicity of Chern-Ricci functions on minimal surfaces in ${\mathbb{R}}^{3}$}]
Let  $\Sigma$ denote a minimal surface immersed in ${\mathbb{R}}^{3}$ with the Gauss curvature $\mathcal{K}={\mathcal{K}}_{ {\mathbf{g}}_{{}_{\Sigma}}}<0$, unit normal vector field $\mathbf{N}$, and  angle function $\mathbf{N}_{\mathbf{V}} \in [-1, 1]$
 with respect to a constant unit vector field ${\mathbf{V}}$ in ${\mathbb{R}}^{3}$. 
\begin{enumerate}
\item When $-1< \mathbf{N}_{\mathbf{V}} \leq 1$, the first Chern-Ricci function ${\textrm{CR}}^{1}_{\mathbf{V}} =\ln \left( \frac{1+ \mathbf{N}_{\mathbf{V}}  }{  {  \left( -  \mathcal{K}  \right) }^{\frac{1}{4}}  } \right)$  is harmonic on $\Sigma$.
\item   When $-1<\mathbf{N}_{\mathbf{V}} <1$, the second Chern-Ricci function ${\textrm{CR}}^{2}_{\mathbf{V}} = \ln \left( \frac{1-{ \mathbf{N}_{\mathbf{V}} }^2 }{  {  \left( -  \mathcal{K}  \right) }^{\frac{1}{2}}  } \right) $  is harmonic on $\Sigma$.
\end{enumerate}
\end{theorem}
  
\begin{proof} 
The key idea is to combine two intriguing identities for Gauss curvature function on negatively curved minimal surfaces in ${\mathbb{R}}^{3}$. On the one 
hand, in his simple proof of Bernstein's Theorem for entire minimal graphs, Chern \cite[Section 4]{Chern1969} used the geometric identity
\begin{center}
       $\mathcal{K} =  {\triangle}_{ {\mathbf{g}}_{{}_{\Sigma}} } \,  \ln \left( 1+ \mathbf{N}_{\mathbf{V}}  \right).$ 
\end{center}
On the other hand, Ricci  \cite{Bl1950, BM2013, Law1970, Law1971, MM2015} obtained the geometric identity 
\begin{center}
       $4  \mathcal{K} =  {\triangle}_{ {\mathbf{g}}_{{}_{\Sigma}} } \,  \ln \left( -   \mathcal{K}  \right), \quad 
  \text{or equivalently,} \quad - \mathcal{K} =  {\triangle}_{ {\mathbf{g}}_{{}_{\Sigma}} } \,  \ln \left(   \frac{1}{ {\left( -   \mathcal{K} \right)}^{\frac{1}{4}}} \right).$
   \label{ricci identity}
\end{center}
Chern's identity and Ricci's identity imply the harmonicity of the first Chern-Ricci function: 
\begin{center}
       ${\triangle}_{ {\mathbf{g}}_{{}_{\Sigma}} } \, {\textrm{CR}}^{1}_{\mathbf{V}} =    {\triangle}_{ {\mathbf{g}}_{{}_{\Sigma}} } \,  \ln \left( \frac{1+ \mathbf{N}_{\mathbf{V}}  }{  {  \left( -  \mathcal{K}  \right) }^{\frac{1}{4}}  } \right) = {\triangle}_{ {\mathbf{g}}_{{}_{\Sigma}} } \,  \ln \left( 1+ \mathbf{N}_{\mathbf{V}}  \right) + {\triangle}_{ {\mathbf{g}}_{{}_{\Sigma}} } \,  \ln \left(   \frac{1}{ {\left( -   \mathcal{K} \right)}^{\frac{1}{4}}} \right)= 0.$ 
\end{center}
The identity $ -  \mathbf{N}_{\mathbf{V}} =  \mathbf{N}_{(\mathbf{-V})}$ and the linear combination of two first Chern-Ricci functions imply
\begin{center}
       ${\triangle}_{ {\mathbf{g}}_{{}_{\Sigma}} } \, {\textrm{CR}}^{2}_{\mathbf{V}} =    {\triangle}_{ {\mathbf{g}}_{{}_{\Sigma}} } \,  
      \ln \left( \frac{1-{ \mathbf{N}_{\mathbf{V}} }^2 }{  {  \left( -  \mathcal{K}  \right) }^{\frac{1}{2}}  } \right)  
       =    {\triangle}_{ {\mathbf{g}}_{{}_{\Sigma}} } \,  \ln \left( \frac{1+ \mathbf{N}_{\mathbf{V}}  }{  {  \left( -  \mathcal{K}  \right) }^{\frac{1}{4}}  } \right)
       +  {\triangle}_{ {\mathbf{g}}_{{}_{\Sigma}} } \,  \ln \left( \frac{1+ \mathbf{N}_{(\mathbf{-V})}  }{  {  \left( -  \mathcal{K}  \right) }^{\frac{1}{4}}  } \right) 
       =  0.$
\end{center}
 \end{proof}

 \section{Classifications of minimal surfaces with constant Chern-Ricci functions}
 
 \begin{theorem}[\textbf{Uniqueness of Enneper's minimal surfaces}]   \label{enneper is unique}
 Let  $\Sigma$ be a minimal surface in ${\mathbb{R}}^{3}$ with the Gauss curvature $\mathcal{K}={\mathcal{K}}_{ {\mathbf{g}}_{{}_{\Sigma}}}<0$ and
unit normal vector $\mathbf{N}$. Suppose that there exists a constant unit vector field ${\mathbf{V}}$ in ${\mathbb{R}}^{3}$ such that $-1<\mathbf{N}_{\mathbf{V}} \leq 1$ and that the first Chern-Ricci function ${\textrm{CR}}^{1}_{\mathbf{V}} =\ln \left( \frac{1+ \mathbf{N}_{\mathbf{V}}  }{  {  \left( -  \mathcal{K}  \right) }^{\frac{1}{4}}  } \right)$ is constant. Then, the minimal surface $\Sigma$ should be a part of Enneper's surface.
 \end{theorem}
 
 \begin{proof} For the simplicity, rotating the coordinate system in ${\mathbb{R}}^{3}$, we can take the normalization ${\mathbf{V}} =-{\mathbf{e}_{3}}= (0, 0, -1)$. We assume that the first Chern-Ricci function 
\begin{equation}  \label{first assumption}
 \ln \left[ \,  {\left( -  \mathcal{K}  \right) }^{ - \frac{1}{4}}   \left( 1+ \mathbf{N}_{\left(-{\mathbf{e}_{3}}\right)}   \right)  \, \right]  = C  
\end{equation} 
is constant. The key idea is to take the orthogonal lines of curvature on our minimal surface $\Sigma$ in order to read the 
information (\ref{first assumption}) in terms of the corresponding Gauss map. We first begin with an arbitrary local conformal 
coordinate $w$ on $\Sigma$ to find the conformal harmonic map $:$
\[
  {\mathbf{X}} =  {\mathbf{X}}(w) =   {\mathbf{X}}({w}_{0}) + \left(  \textrm{Re} \int_{{w}_{0}}^{w} \omega_{1},  \, \textrm{Re} \int_{{w}_{0}}^{w} \omega_{2},  \, \textrm{Re}  \int_{{w}_{0}}^{w}   \omega_{3} \right), 
\]
where the holomorphic $1$-forms $\left(\omega_{1}, \omega_{2}, \omega_{3} \right)$ are given by the Weierstrass data $\left(G(w), \Psi(w) dw\right):$ 
\[
\left(\omega_{1}, \omega_{2}, \omega_{3} \right)=  \left(  \frac{1}{2}  \left(1 -{G(w)}^2 \right)  \Psi(w)\, dw, \,  \frac{i}{2} \left(1 +{G(w)}^2 \right) \Psi(w)
\, dw, \,   {G(w)} \Psi(w) \, dw \right).
\]
We introduce the \textit{isothermic} coordinate $\zeta$ from the initial conformal coordinate $w$ by the rule
\[
  w \mapsto \zeta ={\zeta}_{0} + \int_{w_{0}}^{w} \sqrt{\, G'(w) \Psi(w)  \;} \, dw,
\]
and the mapping ${\mathbf{g}}(\zeta):=G(w)$. Then, the Enneper-Weierstrass representation becomes  
\begin{equation}  \label{in terms of isothermic 1}
  {\mathbf{X}} = {\mathbf{X}}(\zeta) =   {\mathbf{X}}\left({\zeta}_{0}\right) + \left(  \textrm{Re} \int_{{\zeta}_{0}}^{\zeta} \omega_{1},  \, \textrm{Re} \int_{{\zeta}_{0}}^{\zeta} \omega_{2},  \, \textrm{Re}  \int_{{\zeta}_{0}}^{\zeta}   \omega_{3} \right), 
\end{equation}
where the holomorphic $1$-forms $\left(\omega_{1}, \omega_{2}, \omega_{3} \right)$ are given by the Weierstrass data $\left({\mathbf{g}}(\zeta), \frac{1}{
{\mathbf{g}}'(\zeta)} d\zeta\right):$ 
\begin{equation}  \label{one forms 1}
\left(\omega_{1}, \omega_{2}, \omega_{3} \right)=  \left(  \frac{1}{2} \cdot \frac{1 -{({\mathbf{g}}(\zeta))}^2}{  {\mathbf{g}}'(\zeta)}    d\zeta, \,  
 \frac{i}{2} \cdot \frac{ 1 +{({\mathbf{g}}(\zeta))}^2 }{  {\mathbf{g}}'(\zeta)}    d\zeta, \,  \frac{{\mathbf{g}}(\zeta)}{{\mathbf{g}}'(\zeta)} \, d\zeta \right).
\end{equation}
The conformal coordinates $(u, v)=\left(\textrm{Re}\, \zeta, \textrm{Im}\, \zeta\right)$ give us the orthogonal lines of curvature on $\Sigma:$ 
\[
 0 =  \textrm{Im} \left( -  {\mathbf{g}}'(\zeta) \cdot \frac{1}{{\mathbf{g}}'(\zeta)} {d\zeta}^{2}   \right) 
 = \textrm{Im} \left( -  {d\zeta}^{2}   \right) =  - 2 \, du  \, dv.
\]
Using classical formulas (cf. \cite[Chapter 9]{Oss86}), we compute Gauss curvature and angle function:
\begin{equation}  \label{assumption in isothermic}
\mathcal{K}  =   - {\left( \frac{   2 \, \vert   {\mathbf{g}}'(\zeta)   \vert } { \, 1 + {\vert \, {\mathbf{g}}(\zeta) \, \vert}^{2} \, }\right)}^{4} 
\;\; \text{and} \;\;
\mathbf{N}_{\left(-{\mathbf{e}_{3}}\right)}    = \frac{ \, 1 - {\vert \, {\mathbf{g}}(\zeta) \, \vert}^{2} \, }{  \, 1 + {\vert \, {\mathbf{g}}(\zeta) 
\, \vert}^{2} \, }.
\end{equation}
Combining (\ref{first assumption}) and (\ref{assumption in isothermic}) gives 
\begin{equation} \label{first hol}
C = \ln \left[ \,  {\left( -  \mathcal{K}  \right) }^{ - \frac{1}{4}}   \left( 1+ \mathbf{N}_{\left(-{\mathbf{e}_{3}}\right)}   \right)  \, \right]  = 
\ln \left[ \,  \frac{1 + {\vert \, {\mathbf{g}}(\zeta) \, \vert}^{2} }{   2 \, \vert   {\mathbf{g}}'(\zeta)  \vert  } \cdot \frac{2}{ 1 + {\vert \, {\mathbf{g}}(\zeta) 
\, \vert}^{2}} \, \right]=    - \textrm{Re}\, \left[\, \log \, {\mathbf{g}}'(\zeta) \, \right].
\end{equation}
By the holomorphicity of the Gauss map ${\mathbf{g}}(\zeta)$, we have ${\mathbf{g}}(\zeta) = \alpha \left( \zeta - {\zeta}_{0} \right)$. Plugging this into
(\ref{one forms 1}), the Enneper-Weierstrass representation (\ref{in terms of isothermic 1}) shows that $\Sigma$ is Enneper's surface.
 \end{proof}
           
  \begin{remark}
  The associate family deformation of Enneper's surface induces rotations in ${\mathbb{R}}^{3}$.
  \end{remark}
           
 \begin{theorem}[\textbf{Classification of minimal surfaces with constant second Chern-Ricci function}]   \label{second classification}
 Let  $\Sigma$ be a minimal surface in ${\mathbb{R}}^{3}$ with the Gauss curvature $\mathcal{K}={\mathcal{K}}_{ {\mathbf{g}}_{{}_{\Sigma}}}<0$ and
unit normal vector $\mathbf{N}$. Suppose that there exists a constant unit vector field ${\mathbf{V}}$ in ${\mathbb{R}}^{3}$ such that $-1<\mathbf{N}_{\mathbf{V}}< 1$ and that the second Chern-Ricci function ${\textrm{CR}}^{2}_{\mathbf{V}} =\ln \left(  { \left( -  \mathcal{K}  \right) }^{ - \frac{1}{2}}   \left(  1 -  { \mathbf{N}_{\mathbf{V}} }^2 \right)    \right)$ is constant. Then, $\Sigma$ is a part of the minimal surface given by
the Enneper-Weierstrass representation    
\begin{equation}  \label{in terms of isothermic 2}
{\mathbf{X}}(\zeta) =   {\mathbf{X}}\left({\zeta}_{0}\right) + \left(  \textrm{Re} \int_{{\zeta}_{0}}^{\zeta} \omega_{1},  \, \textrm{Re} \int_{{\zeta}_{0}}^{\zeta} \omega_{2},  \, \textrm{Re}  \int_{{\zeta}_{0}}^{\zeta}   \omega_{3} \right)
\end{equation}
where the holomorphic $1$-forms 
\begin{equation}  \label{one forms 2}
\left(\omega_{1}, \omega_{2}, \omega_{3} \right)=  \left(  \frac{1}{2} \cdot \frac{1 -{({\mathbf{g}}(\zeta))}^2}{  {\mathbf{g}}'(\zeta)}    d\zeta, \,  
 \frac{i}{2} \cdot \frac{ 1 +{({\mathbf{g}}(\zeta))}^2 }{  {\mathbf{g}}'(\zeta)}    d\zeta, \,  \frac{{\mathbf{g}}(\zeta)}{{\mathbf{g}}'(\zeta)} \, d\zeta \right).
\end{equation}
are given by the Weierstrass data 
$\left({\mathbf{g}}(\zeta), \frac{1}{{\mathbf{g}}'(\zeta)} d\zeta\right)=\left( e^{\alpha \left(\zeta - {\zeta}_{0} \right)},  e^{-\alpha \zeta}  d \zeta \right)$
for some constants $\alpha \in {\mathbb{C}}-\{0\}$ and ${\zeta}_{0} \in {\mathbb{C}}$.
\end{theorem}
 
  \begin{proof} For the simplicity, rotating the coordinate system in ${\mathbb{R}}^{3}$, we can take the normalization ${\mathbf{V}} ={\mathbf{e}_{3}}= (0, 0, 1)$. As in the proof of Theorem \ref{enneper is unique}, taking the orthogonal lines of curvature coordinates $\zeta$ on the minimal surface under the normalization of Hopf differential $-{d{\zeta}}^{2}$, we can write the second Chern-Ricci function in terms of the corresponding Gauss map 
${\mathbf{g}}(\zeta)$: 
\begin{center}
${\textrm{CR}}^{2}_{{\mathbf{e}_{3}}} = 
 \ln \left(  { \left( -  \mathcal{K}  \right) }^{ - \frac{1}{2}}   \left(  1 -  { \mathbf{N}_{ {\mathbf{e}_{3}} } }^2 \right)    \right)
 = \ln \left[ \, {\left( \frac{1 + {\vert \, {\mathbf{g}}(\zeta) \, \vert}^{2} }{   2 \, \vert   {\mathbf{g}}'(\zeta)  \vert  } \right)}^{2} 
{\left( \frac{2 \vert   {\mathbf{g}} (\zeta)  \vert}{ 1 + {\vert \, {\mathbf{g}}(\zeta) \, \vert}^{2}}  \right)}^{2} \, \right]=   
 2 \, \ln \left( \frac{ \vert {\mathbf{g}}(\zeta) \vert }{ \vert {\mathbf{g}}'(\zeta) \vert } \right),$
\end{center}
or equivalently,
\begin{equation} \label{second hol}
 {\textrm{CR}}^{2}_{{\mathbf{e}_{3}}}  = - 2 \, \textrm{Re}\, \left[ \, \log \,  \frac{{\mathbf{g}}'(\zeta)}{{\mathbf{g}}(\zeta)} \, \right] 
 = - 2 \, \textrm{Re}\, \left[ \, \log \,  {\left(\,  \log  \,   {\mathbf{g}}(\zeta) \,   \right)}' \, \right]. 
\end{equation}
Since the function ${\textrm{CR}}^{2}_{{\mathbf{e}_{3}}}$  is constant, by the holomorphicity of ${\mathbf{g}}(\zeta)$, we have the 
Weierstrass data 
\[
\left({\mathbf{g}}(\zeta), \frac{1}{{\mathbf{g}}'(\zeta)} d\zeta\right)=\left( e^{\alpha \left(\zeta -{\zeta}_{0} \right)},  e^{-\alpha \zeta}  d \zeta \right).
\]  
\end{proof}

\begin{remark}[\textbf{Examples of minimal surfaces with constant second Chern-Ricci function}]
The moduli space in Theorem \ref{second classification} contains catenoids \cite[Section 2.1.2]{BLN2004} and helicoids  \cite[Section 2.1.4]{BLN2004}. In fact, under the transformation 
$\zeta \mapsto z :=  e^{\alpha \left(\zeta - {\zeta}_{0} \right)}$, we obtain the Weierstrass data 
\begin{center}
  $
\left({\mathbf{g}}(\zeta), \frac{1}{{\mathbf{g}}'(\zeta)} d\zeta\right)=\left( e^{\alpha \left(\zeta - {\zeta}_{0} \right)},  e^{-\alpha \zeta}  d \zeta \right)
= \left(z,  \frac{e^{ - \alpha   {\zeta}_{0}}}{\alpha} \cdot \frac{1}{z^2} dz \right)$, 
\end{center}
which recovers the associate family of helicoids. 
\end{remark}

\section*{Appendix. Four holomorphic quadratic differentials on minimal surfaces in ${\mathbb{R}}^{3}$}
  
We summarize definitions of holomorphic quadratic differentials on minimal surfaces in ${\mathbb{R}}^{3}$.
 Throughout this section, as in the proofs of Theorem \ref{enneper is unique} and \ref{second classification}, we use the orthogonal lines of curvature 
 coordinates $\zeta$ on the negatively curved minimal surface $\Sigma$ with the normalization of Hopf's holomorphic 
 differential $  {\mathcal{Q}} = - \frac{1}{2} {d{\zeta}}^{2}$.  The minimal surface $\Sigma$ is parameterized by     
\begin{center}
  ${\mathbf{X}}(\zeta) =   {\mathbf{X}}\left({\zeta}_{0}\right) + \left(  \textrm{Re} \int_{{\zeta}_{0}}^{\zeta} \omega_{1},  \, \textrm{Re} \int_{{\zeta}_{0}}^{\zeta} \omega_{2},  \, \textrm{Re}  \int_{{\zeta}_{0}}^{\zeta}   \omega_{3} \right)$, 
\end{center}
where we have the holomorphic $1$-forms
$\left(\omega_{1}, \omega_{2}, \omega_{3} \right)=  \left(  \frac{1}{2} \cdot \frac{1 -{({\mathbf{g}}(\zeta))}^2}{  {\mathbf{g}}'(\zeta)} , \,  
 \frac{i}{2} \cdot \frac{ 1 +{({\mathbf{g}}(\zeta))}^2 }{  {\mathbf{g}}'(\zeta)}  , \,  \frac{{\mathbf{g}}(\zeta)}{{\mathbf{g}}'(\zeta)} \,  \right)   d\zeta$.
 
\subsection*{I. Bernstein-Mettler's entropy differential \cite{BM2013}}
Applying the variational structure of Ricci's intrinsic condition induced from Hamilton's
  entropy functional for Ricci flow, Bernstein and Mettler constructed the entropy  differential. In \cite[Corollary A.2]{BM2013},
they observed that, if the minimal surface $\Sigma$ is an Enneper surface, the conformally changed metric 
${\left( -  {\mathcal{K}}_{{}_{\textrm{\textbf{Enn}}}} \right)}^{\frac{3}{4}}  {\mathbf{g}}_{{}_{\textrm{\textbf{Enn}}}}$ recovers the positively curved cigar soliton, and also proved that this property characterizes Enneper surfaces among minimal surfaces in ${\mathbb{R}}^{3}$.  
  The formula in \cite[Proposition 3.2]{BM2013} implies that the Schwarzian derivative of the  Gauss map  ${\mathbf{g}}(\zeta)$ realizes the holomorphic quadratic differential
\begin{equation} \label{holomorphic differential 3}
   \mathbf{S}\mathbf{g} (\zeta) \, d{\zeta}^2 := \left[   {\left( \frac{  {\mathbf{g}}''(\zeta)  }{   {\mathbf{g}}'(\zeta)} \right)}'   
  - \frac{1}{2} {\left(  {\frac{  {\mathbf{g}}''(\zeta)  }{   {\mathbf{g}}'(\zeta)} } \right)}^{2}   \,  \right]   d{\zeta}^{2}. 
\end{equation}
We would like to add that the Schwarzian derivative of the gauss map realizes the squared 
complex curvature of the lifted holomorphic null curve from the minimal surface \cite[Section 3]{Go1999}. 
 
\subsection*{II. Induced holomorphic quadratic differential from Chern-Ricci functions}
By the identity (\ref{first hol}), the holomorphicity of the corresponding Gauss map  ${\mathbf{g}}(\zeta)$ also implies the harmonicity of
  the first Chern-Ricci function
\[
{\textrm{CR}}^{1}_{\left(-{\mathbf{e}_{3}}\right)} =\ln \left[ \,  {\left( -  \mathcal{K}  \right) }^{ - \frac{1}{4}}   \left( 1+ \mathbf{N}_{\left(-{\mathbf{e}_{3}}\right)}   \right)  \, \right]    =    - \textrm{Re}\, \left[\, \log \, {\mathbf{g}}'(\zeta)\, \right].
\]
The harmonic function $-{\textrm{CR}}^{1}_{\left(-{\mathbf{e}_{3}}\right)}$ induces the holomorphic quadratic differential
\begin{equation} \label{holomorphic differential 1}
  {\mathcal{Q}}_{1}  =  \left( \log \, {\mathbf{g}}'(\zeta) \right)'' d{\zeta}^2  =  {\left( \frac{  {\mathbf{g}}''(\zeta)  }{   {\mathbf{g}}'(\zeta)} \right)}'    d{\zeta}^{2}. 
\end{equation}
By the identity (\ref{second hol}),  the holomorphicity of the Gauss map  ${\mathbf{g}}(\zeta)$ also implies the harmonicity of
  the second Chern-Ricci function
\[
{\textrm{CR}}^{2}_{{\mathbf{e}_{3}}}  = \ln \left(  { \left( -  \mathcal{K}  \right) }^{ - \frac{1}{2}}   \left(  1 -  { \mathbf{N}_{ {\mathbf{e}_{3}} } }^2 \right)    \right)
 = - 2 \, \textrm{Re}\, \left[ \, \log \,  {\left(\,  \log  \,   {\mathbf{g}}(\zeta) \,   \right)}' \, \right]. 
 \]
 The harmonic function $- \frac{1}{2} {\textrm{CR}}^{2}_{{\mathbf{e}_{3}}}$ induces the holomorphic quadratic differential
\begin{equation} \label{holomorphic differential 2}
  {\mathcal{Q}}_{2}  =  \left( \log \,  {\left(\,  \log  \,   {\mathbf{g}}(\zeta) \,   \right)}'  \right)'' d{\zeta}^2  
 =  \left[   {\left( \frac{  {\mathbf{g}}''(\zeta)  }{   {\mathbf{g}}'(\zeta)} \right)}'   
  - {\left(  {\frac{  {\mathbf{g}}'(\zeta)  }{   {\mathbf{g}}(\zeta)} } \right)}'   \,  \right]   d{\zeta}^2. 
\end{equation}

\bigskip

\end{document}